\documentclass[11pt, a4paper, reqno, oneside]{amsart}
\usepackage{amsmath, amssymb, amsfonts, amsthm, overpic, bm, url, graphicx}

\newcommand{\epsi}{\varepsilon}

\newcommand{\numbersystem}[1]{\mathbb{#1}}

\newcommand{\R}{\numbersystem{R}}

\newcommand{\abs}[1]{\left\lvert#1\right\rvert}

\newcommand{\myangle}{\sphericalangle}
\newcommand{\length}[1]{\left\lvert#1\right\rvert}
\newcommand{\card}[1]{\lvert#1\rvert}

\newcommand{\define}[1]{\emph{#1}}

\theoremstyle{plain}
\newtheorem{theorem}{Theorem}
\newtheorem{lemma}[theorem]{Lemma}
\newtheorem{corollary}[theorem]{Corollary}
\newtheorem{proposition}[theorem]{Proposition}

\newtheorem*{stabilitytheorem}{Stability Theorem}

\theoremstyle{definition}

\begin{document}

\bibliographystyle{amsplain}

\title{Unit distances and diameters in Euclidean spaces}
\author{Konrad J. Swanepoel}
\thanks{This material is based upon work supported by the South African National Research Foundation.}
\address{Department of Mathematical Sciences,
        University of South Africa, PO Box 392,
        Pretoria 0003, South Africa}
\email{\texttt{swanekj@unisa.ac.za}}

\begin{abstract}
We show that the maximum number of unit distances or of diameters in a set of $n$ points in $d$-dimensional Euclidean space is attained only by specific types of Lenz constructions, for all $d\geq 4$ and $n$ sufficiently large, depending on $d$.
As a corollary we determine the exact maximum number of unit distances for all even $d\geq 6$, and the exact maximum number of diameters for all $d\geq 4$, for all $n$ sufficiently large, depending on $d$.
\end{abstract}
\maketitle

\section{Introduction}

\subsection{Unit distances}
For a finite subset $S$ of Euclidean $d$-space $\R^d$ let $u(S)$ denote the number of pairs of points in $S$ at distance $1$.
Define \[u_d(n) = \max \{u(S): S\subset\R^d, \card{S}=n\}. \]

Erd\H{o}s initiated the study of $u_2(n)$ in \cite{Erdos46} and the higher-dimensional case of $u_d(n)$, $d\geq 3$, in \cite{Erdos60}.
The cases $d=2$ and $d=3$ are the most difficult.
Erd\H{o}s \cite{Erdos46} obtained the superlinear lower bound
\[ u_2(n) \geq n^{1+\frac{c}{\log\log n}},\]
which he conjectured to be tight \cite{MR80i:05001, MR88f:52011, MR92f:11003, MR1476428}.
The best known upper bound is \[ u_2(n) \leq cn^{4/3},\]
due to Spencer, Szemer\'edi and Trotter \cite{SST}.
See Sz\'ekely \cite{Szekely} for a particularly simple proof.

For $d=3$ the known lower (Erd\H{o}s \cite{Erdos60}) and upper bounds (Clarkson et al.\ \cite{Clarksonetal}) are:
\[ cn^{4/3}\log\log n \leq u_3(n) \leq  cn^{3/2}\beta(n),\]
where $\beta(n)$ is an extremely slowly growing function related to the inverse Ackerman function.

For $d\geq 4$ (the subject of this paper) the situation changes drastically.
Lenz, as reported in \cite{Erdos60}, observed that if we take $p:=\lfloor d/2\rfloor$ circles in pairwise orthogonal $2$-dimensional subspaces, each with centre the origin and radius $1/\sqrt{2}$, then any two points on different circles are at unit distance.
Therefore, if $n$ points are chosen by taking $n/p+O(1)$ points on each circle, $\frac{p-1}{2p}n^2-O(1)$ unit distances are obtained.
Erd\H{o}s \cite{Erdos60} showed that since $K_{p+1}(3)$, the complete $(p+1)$-partite graph with three vertices in each class, does not occur as a unit-distance graph in $\R^d$, the Erd\H{o}s-Stone theorem gives:
\[ u_d(n) = \frac{p-1}{2p}n^2+o(n^2) \text{ for all } d\geq 4.\]

Using an extremal graph theory result of Erd\H{o}s \cite{Erdos63} and Simonovits \cite{Simonovits}, Erd\H{o}s \cite{Erdos67} determined the exact value of $u_d(n)$ for $d$ even and $n$ a sufficiently large (depending on $d$) multiple of $2d=4p$.
The $n/p$ points on each circle are then taken to be the vertices of $n/(4p)$ squares.
This determines $u_d(n)$ asymptotically for all sufficiently large $n$ up to a $O(1)$ term (still for $d$ even).
Brass \cite{Brass}, together with a number theoretical result of Van Wamelen \cite{VanWamelen}, determined $u_4(n)$ completely.
For $n\geq 5$,
\begin{align*}
u_4(n) &= \begin{cases}
\lfloor n^2/4\rfloor + n & \text{if $n$ is divisible by $8$ or $10$,}\\
\lfloor n^2/4\rfloor + n - 1 & \text{otherwise.}
\end{cases}
\end{align*}

For odd $d\geq 5$ Erd\H{o}s and Pach \cite{ErdosPach} showed that
\[ u_d(n) = \frac{p-1}{2p}n^2 + \Theta(n^{4/3}). \]
For the lower bound they observed that the Lenz construction can be improved when $d$ is odd by replacing one of the circles by a $2$-sphere of radius $1/\sqrt{2}$ in a $3$-dimensional space orthogonal to the other $2$-dimensional subspaces and by placing the points on the sphere such that the unit distance occurs at least $cn^{4/3}$ times (a construction of Erd\H{o}s, Hickerson and Pach \cite{EHP}).
For the upper bound they used a stability result in extremal graph theory \cite[Chapter 5, remark 4.5(ii)]{Bollobas} together with the fact that the maximum number of unit distances among $n$ points on a $2$-sphere is $O(n^{4/3})$ \cite{Clarksonetal}.

\subsection{Diameters}
For a finite subset $S$ of $\R^d$  we call a pair of points in $S$ a \define{diameter} if their distance equals the diameter of $S$.
Let $M(S)$ denote the number of diameters in $S$.
Define \[M_d(n) = \max \{M(S): S\subset\R^d, \card{S}=n\}. \]
Erd\H{o}s in \cite{Erdos46} showed that $M_2(n)=n$ for $n\geq 3$.
V\'azsonyi conjectured, as reported in \cite{Erdos46}, that $M_3(n)=2n-2$ for $n\geq 4$.
This was independently proved by Gr\"unbaum \cite{Grunbaum}, Heppes \cite{Heppes} and Straszewicz \cite{Str}.
For a new proof, see \cite{Sw-Vazsonyi}.

As in the case of unit distances, the situation is completely different when $d\geq 4$.
Erd\H{o}s \cite{Erdos60} showed that for $d\geq 4$, $M_d(n)=\frac{p-1}{2p}n^2+o(n^2)$, the same asymptotics as $u_d(n)$.
For other work on this problem by Hadwiger, Lenz and Yugai, see the survey of Martini and Soltan \cite{MS}.

\section{New results}\label{section2}
If $d\geq 4$ is even, let $p=d/2$ and consider any orthogonal decomposition $\R^d=V_1\oplus+\dots\oplus V_p$, where each $V_i$ is $2$-dimensional.
In each $V_i$, let $C_i$ be the circle with centre the origin $o$ and radius $r_i$ such that $r_i^2+r_j^2=1$ for all distinct $i$ and $j$.
When $d\geq 6$ this implies that each $r_i=1/\sqrt{2}$.
We define a \define{Lenz configuration} to be any translate of a finite subset of $\bigcup_{i=1}^p C_i$.

If $d\geq 5$ is odd, let $p=\lfloor d/2\rfloor$, and consider any orthogonal decomposition $\R^d=V_1\oplus\dots\oplus V_p$, where $V_1$ is $3$-dimensional and each $V_i$ ($i=2,\dots,p$) is $2$-dimensional.
Let $\Sigma$ be the sphere in $V_1$ with centre $o$ and radius $r_1$, and for each $i=2,\dots,p$, let $C_i$ be the circle with centre $o$ and radius $r_i$, such that $r_i^2+r_j^2=1$ for all distinct $i, j$.
When $d\geq 7$, necessarily each $r_i=1/\sqrt{2}$.
We define a \define{Lenz configuration} to be any translate of a finite subset of $\Sigma\cup\bigcup_{i=2}^p C_i$.
(Later we distinguish between \emph{weak} and \emph{strong} Lenz configurations as a technical notion inside the proofs. The definition here coincides with a strong Lenz construction in the sequel).

We call a set $S$ of $n$ points in $\R^d$ an \define{extremal set} with respect to unit distances [diameters] if $u(S)=u_d(n)$ [$M(S)=M_d(n)$].

\begin{theorem}\label{maintheorem}
For each $d\geq 4$ there exists $N(d)$ 
such that all extremal sets of $n\geq N(d)$ points 
\textup{(}with respect to unit distances or 
diameters\textup{)} are Lenz configurations.
\end{theorem}

The proof uses a typical technique in extremal graph and hypergraph theory \cite{Simonovits, FS, KS, MP}:
First prove a stability result for sets that are close to extremal, and then deduce more exact structural information from extremality.

For even $d\geq 6$ it is then possible to determine $u_d(n)$ exactly.
On the other hand, for odd $d\geq 5$ the main obstacle to determine $u_d(n)$ is our lack of knowledge of the function $f(m)$ which gives the exact maximum number of unit distances between $m$ points on a $2$-sphere of radius $1/\sqrt{2}$ (for odd $d\geq 7$) and the function $g(m)$ which gives the exact maximum number of unit distances between $m$ points on a sphere of arbitrary radius \cite{EHP, SV} (for $d=5$).

Let \define{$t_p(n)$} denote the number of edges of the \emph{Tur\'an $p$-partite graph} on $n$ vertices.
This is the complete $p$-partite graph with $\lfloor n/p\rfloor$ or $\lceil n/p\rceil$ in each class \cite[Chapter VI]{Bollobas}.
We do not need the exact value of $t_p(n)$, only that
\[ t_p(n) = \frac{p-1}{2p}n^2 -O(1). \]
\begin{corollary}\label{cor2}
Let $d\geq 6$ be even.
For all sufficiently large $n$ \textup{(}depending on $d$\textup{)},
\begin{align*}
u_d(n) &= \begin{cases}
t_p(n)+n-r & \text{if}\quad 0\leq r\leq p-1,\\
t_p(n)+n-p & \text{if}\quad p\leq r\leq 3p-1,\\
t_p(n)+n-2d+r & \text{if}\quad 3p\leq r\leq 4p-1,
\end{cases}
\end{align*}
where $p=d/2$ and $r$ is the remainder when dividing $n$ by $4p=2d$.
\end{corollary}

For all $d\geq 4$ it is possible to determine $M_d(n)$ exactly if $n$ is large.
The most complicated case is $d=5$, where it is necessary to know the maximum number of diameters in a set of $n$ points on a $2$-sphere in $\R^3$.
For each $n\geq 6$ we construct a set of $n$ points in $\R^3$ with $2n-2$ diameters, all lying on a sphere (see Lemma~\ref{lowdim}\eqref{d} below).

\begin{corollary}\label{cor3}
For all sufficiently large $n$ \textup{(}depending on $d$\textup{)},
\begin{align*}
M_4(n) &= \begin{cases}
t_2(n)+\lceil n/2\rceil + 1 & \text{if}\quad n\not\equiv 3\pmod{4},\\
t_2(n)+\lceil n/2\rceil    & \text{if}\quad n\equiv 3\pmod{4};
\end{cases}\\
M_5(n) &= t_2(n) +n;\\
M_d(n) &= t_p(n) + p\quad\text{for even $d\geq 6$, where $p=d/2$;}\\
M_d(n) &= t_p(n) + \lceil n/p\rceil +p-1\quad\text{for odd $d\geq 7$, where $p=\lfloor d/2\rfloor$.}
\end{align*}
\end{corollary}

We use two stability theorems to prove Theorem~\ref{maintheorem}, one for even dimensions and one for odd dimensions.
\begin{theorem}\label{evenstable}
For each $\epsi>0$ and even $d\geq 4$ there exist $\delta>0$ and $N$ such that any set of $n\geq N$ points in $\R^d$ with at least $(\frac{p-1}{2p}-\delta)n^2$ unit distance pairs can be partitioned into $S_0,S_1,\dots,S_p$ such that $\card{S_0}<\epsi n$ and for each $i=1,\dots,p$,
\[ \frac{n}{p}-\epsi n < \card{S_i} < \frac{n}{p}+\epsi n\]
and $S_i$ is on a circle $C_i$, such that the circles $C_1,\dots,C_p$ have the same centre and are mutually orthogonal.
\end{theorem}

\begin{theorem}\label{oddstable}
For each $\epsi>0$ and odd $d\geq 5$ there exist $\delta>0$ and $N$ such that any set $S$ of $n\geq N$ points in $\R^d$ with at least $(\frac{p-1}{2p}-\delta)n^2$ unit distance pairs can be partitioned  into $S_0,S_1,\dots,S_p$ such that $\card{S_0}<\epsi n$ and for each $i=1,\dots,p$,
\[ \frac{n}{p}-\epsi n < \card{S_i} < \frac{n}{p}+\epsi n,\]
$S_1$ is on a $2$-sphere $\Sigma_1$, $S_i$ is on a circle $C_i$, $i=2,\dots,p$, and $\Sigma_1,C_2,\dots,C_p$ have the same centre and are mutually orthogonal.
\end{theorem}

\begin{corollary}\label{stablecor} Let $d\geq 4$.
If a set $S$ of $n$ points in $\R^d$ has at least $(\frac{p-1}{2p}-o(1))n^2$ unit distance pairs, then $S$ is a Lenz configuration except for $o(n)$ points.
\end{corollary}

\section{Overview of the paper}
In Section~\ref{lemmas} we consider results from geometry necessary for the proofs.

In Section~\ref{optimised} we determine the maximum number of unit distances and diameters in even-dimensional Lenz configurations, introduce the notions of weak and strong Lenz configuration in odd dimensions, show that the weak Lenz configurations with the largest number of unit distances or diameters are strong Lenz configurations, and determine the maximum number of diameters in strong Lenz configurations.
Corollaries~\ref{cor2} and \ref{cor3} then follow, given that extremal sets are (weak) Lenz configurations.

In Section~\ref{stability} we use the Erd\H{o}s-Simonovits stability
theorem from extremal graph theory to prove Theorems~\ref{evenstable}
and \ref{oddstable}, from which Corollary~\ref{stablecor} is
immediate.

Finally, in Section~\ref{extremal} we use the stability theorems to show that sets of points that are extremal with respect to unit distances or diameters are (weak) Lenz configurations, thereby finishing the proof of Theorem~\ref{maintheorem}.

\section{Geometric preliminaries}\label{lemmas}
We denote the distance between points $p$ and $q$ in $\R^d$ by $\length{pq}$.
The \define{unit distance graph} of a set $S$ of $n$ points in $\R^d$ is defined by joining any two points at distance $1$.
Let $u(S)$ denote the number of (unordered) unit distance pairs in $S$.
Two points in $S$ at distance $1$ are \define{neighbours}.
For any point $x$ and finite set $S$, let $u(x,S)$ denote the number of points in $A$ that are at distance $1$ to $x$.
Similarly, for any finite sets $A$ and $B$, let $u(A,B)$ denote the number of (ordered) unit distance pairs $(a,b)$ with $a\in A$ and $b\in B$.

Whenever we work with diameters, we assume that the diameter of $S$ is $1$, and then we use the notation $u(S)$, $u(x,S)$ and $u(A,B)$ as before.
In this case we call the unit distance graph of $S$ the \define{diameter graph} of $S$.

We continually use the following two basic lemmas in the sequel.
The first deals with unit distances and diameters on circles and $2$-spheres, and the second with unit distances in dimensions higher than $3$.

\begin{lemma}\label{lowdim}
Let $S$ be a set of $n$ points in $\R^3$.
\begin{enumerate}
\renewcommand{\theenumi}{\alph{enumi}} 
\item\label{a} If $S$ lies on a circle of radius $1/\sqrt{2}$, then
\[ u(S) \leq \begin{cases}
n & \text{if $n$ is divisible by $4$,}\\
n-1 & \text{otherwise.}
\end{cases} \]
Equality is possible for all $n$, by letting $S$ be the union of the vertices of $\lfloor n/4\rfloor$ inscribed squares and $n-4\lfloor n/4\rfloor$ vertices  of an additional square.
\item\label{b} If $S$ has diameter $1$ and lies on a circle, then
\[ u(S) \leq \begin{cases}
n & \text{if $n$ is odd,}\\
n-1 & \text{if $n$ is even.}
\end{cases} \]
Equality is possible for all $n\geq 2$, for a circle of suitable radius depending on $n$.
\item\label{c} If $S$ has diameter $1$ and lies on a circle of radius $>1/\sqrt{3}$, then $u(S)=1$.
\item\label{new} If $S$ lies on a $2$-sphere, then $u(S)=O(n^{4/3})$.
There exist sets $S$ with $u(S)=\Omega(n^{4/3})$.
\item\label{d} If $S$ has diameter $1$ and lies on a $2$-sphere, then $u(S)\leq 2n-2$.
Equality is possible for each $n\geq 4$, $n\neq 5$, for a $2$-sphere of suitable radius depending on $n$.
\item\label{e} If $S$ has diameter $1$ and lies on a $2$-sphere of radius $\geq 1/\sqrt{2}$, then $u(S)\leq n$.
Equality is possible for all $n\geq 3$ and all radii $\geq 1/\sqrt{2}$.
\end{enumerate}
\end{lemma}

\begin{proof}
Statements \eqref{a}, \eqref{b}, \eqref{c} are straightforward, except perhaps $u(S)\leq n-1$ for an even number of concyclic points of diameter $1$.
This follows essentially from the easily seen observation that if the diameter graph of points of some concyclic points contains a cycle, then it consists only of this cycle, together with the well known fact that all cycles in diameter graphs in the plane are odd \cite{HP, Sutherland}.

The upper bound in \eqref{new} is due to Clarkson et al.\ \cite{Clarksonetal}.
The simplest known proof of it is by adapting Sz\'ekely's proof \cite{Szekely} for the planar case.
The lower bound in \eqref{new} is due to Erd\H{o}s, Hickerson and Pach \cite{EHP}.

Statement \eqref{e} can be found in Kupitz, Martini and Wegner \cite{KMW}.
It follows as in the planar case \cite[Theorem 13.13]{PachAgarwal} from the observation that any two diameters, when drawn as short great circular arcs on the $2$-sphere, must intersect.
Examples of $n$ points with $n$ diameters are easily found for all radii larger than $1/\sqrt{2}$; they have essentially the same structure as in the plane; see \cite{KMW} for details.

The upper bound of $2n-2$ in \eqref{d} is the Gr\"unbaum-Heppes-Straczewicz upper bound for diameters in $\R^3$ \cite[Theorem 13.14]{PachAgarwal}.
(For a new proof see \cite{Sw-Vazsonyi}.)
The following is a short proof for points on a $2$-sphere.
For a point $x$ on the sphere, denote its opposite point by $x'$.
Colour the $n$ given points blue and their opposite points red.
For any diameter $xy$, join the blue point $x$ and the red point $y'$ by a short arc of the great circle passing through them, and do the same with $x'$ and $y$.
This defines a bipartite geometric graph on the sphere, with all the arcs of the same length $r$, say.
It is easily seen that this graph is planar: if the arcs $ab'$ and $cd'$ intersect, then by the triangle inequality, the arc $ad'$ or the arc $b'c$ will be shorter than $r$.
Then either $\length{ad}$ or $\length{bc}$ will be larger than the diameter, a contradiction.
This graph has $2n$ vertices.
By Euler's formula, a bipartite planar graph on $2n$ vertices has at most $4n-4$ edges.
Since this is twice the number of diameters, the upper bound follows.

The only statement that remains to be proved, is that $2n-2$ diameters can be attained on a $2$-sphere for each $n\geq 4$, $n\neq 5$.
For even $n\geq 4$ the construction is easy.
Consider the vertex set of a regular $(n-1)$-gon of diameter $1$, and choose another point on the axis of symmetry of the polygon at distance $1$ to the $n-1$ vertices.
This clearly gives $n$ points with $2n-2$ diameters.

For odd $n\geq 7$ the construction is more involved.
Place $n-3$ points $x_1,\dots,x_{n-3}$ on the circle $C$ of radius $r$ and centre $o$ in the $xy$-plane such that the diameter $1$ occurs between consecutive $x_i$'s (Figure~\ref{fig1}).
\begin{figure}
\begin{center}
\begin{overpic}[scale=0.7]{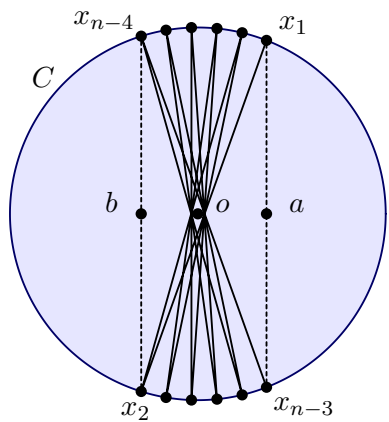}
\put(51,48){$o$}
\put(65,48){$a$}
\put(30,48){$b$}
\put(63,83){$x_1$}
\put(33,9){$x_2$}
\put(62,10){$x_{n-3}$}
\put(24,84){$x_{n-4}$}
\put(16,72){$C$}
\end{overpic}
\end{center}
\caption{Circle $C$ with points $x_1$ to $x_{n-3}$}\label{fig1}
\end{figure}
Note that $r$ and $n$ determine everything up to isometry.
We fix $r$ later in the proof.
Let $x_{n-2}$ be the point on the positive $z$-axis at distance $1$ to $C$.
Then $x_{n-2}$ and $C$ are on a unique sphere $\Sigma$ with centre $o'$ and radius $s$, say.
Note that $o'$ is on the positive $z$-axis.

We now want to find points $x_{n-1}$ and $x_n$ on $\Sigma$ 
such that
\[ \length{x_1x_{n-1}}=\length{x_{n-3}x_{n-1}}=\length{x_2x_n}=\length{x_{n-4}x_n}=\length{x_{n-1}x_n}=1\]
and \[\length{x_{n-2}x_{n-1}}\leq 1, \quad\length{x_{n-2}x_n}\leq 1.\]See Figure~\ref{fig2}.
\begin{figure}
\begin{center}
\begin{overpic}[scale=0.7]{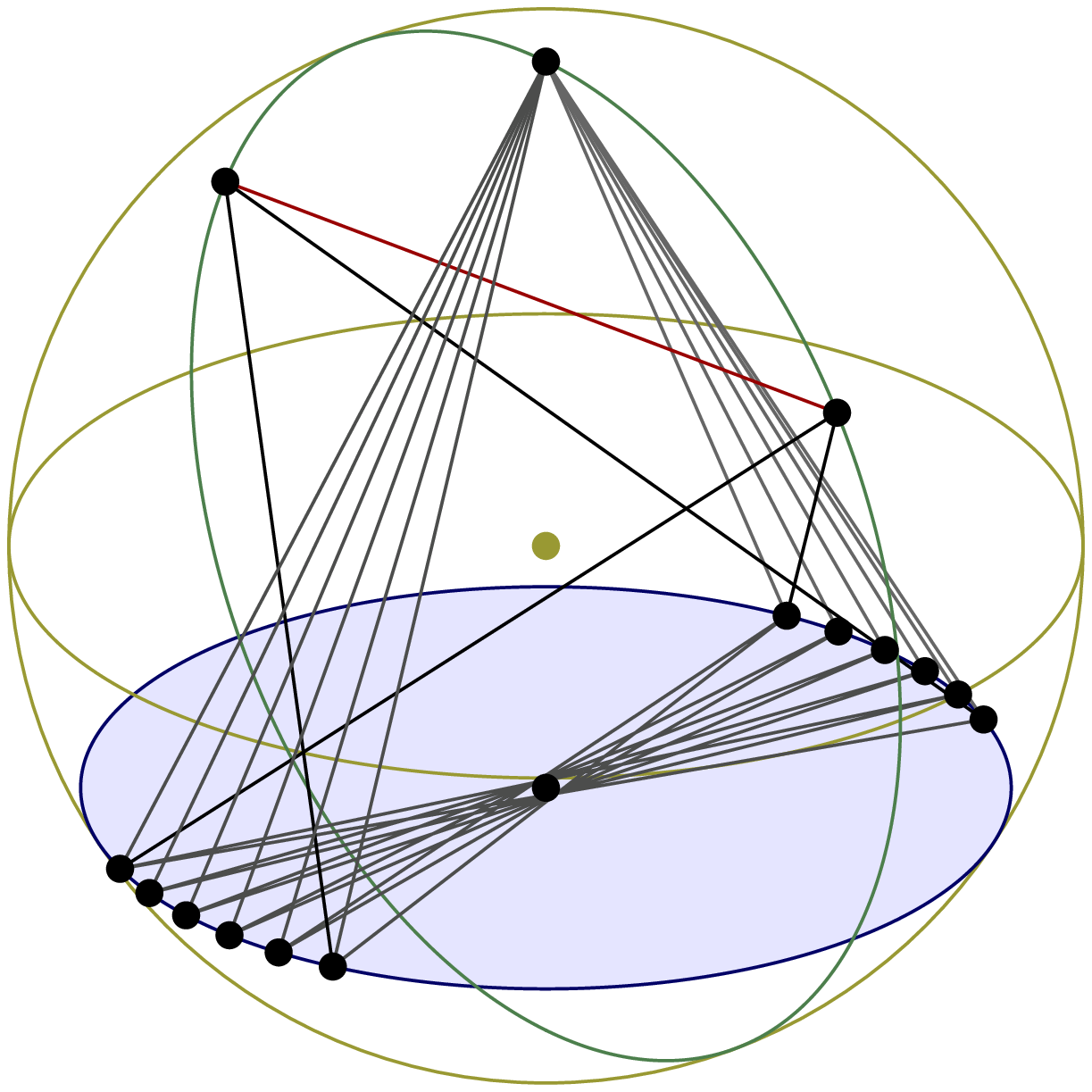}
\put(51,27){$o$}
\put(51,51){$o'$}
\put(30,12){$x_{n-3}$}
\put(82,32){$x_1$}
\put(9,23){$x_2$}
\put(57,43){$x_{n-4}$}
\put(27,80){$x_{n-1}$}
\put(75,60){$x_n$}
\put(52,88){$x_{n-2}$}
\put(60,16){$C$}
\put(18,49){$C'$}
\put(82,80){$\Sigma$}
\end{overpic}
\end{center}
\caption{$15$ points on a sphere with $28$ diameters}\label{fig2}
\end{figure}
This will give the required number of diameters in the set $S:=\{x_1,\dots,x_n\}$.
For any value of $r$ there will clearly be unique points $x_{n-1},x_n\in\Sigma\setminus\{x_{n-2}\}$ that satisfy
\[\length{x_{n-3}x_{n-1}}=\length{x_1x_{n-1}}=\length{x_2x_n}=\length{x_{n-4}x_n}=1.\] 
It remains to find an appropriate value of $r$ so that
\[\length{x_{n-1}x_n}=1, \quad\length{x_{n-2}x_{n-1}}\leq 1, \quad\length{x_{n-2}x_n}\leq 1.\]

We reduce this to a two-dimensional problem.
Let $a$ and $b$ be the midpoints of $x_1x_{n-3}$ and $x_2x_{n-4}$, respectively.
Consider the intersection of $\Sigma$ with the plane $oabx_{n-2}$.
This is a circle $C'$ with centre $o'$ and radius $s$.
By symmetry, $x_{n-1}$ and $x_n$ lie on $C'$, and $\length{ax_{n-2}}=\length{ax_{n-1}}$ and $\length{bx_{n-2}}=\length{bx_n}$ (Figure~\ref{fig3}).
\begin{figure}
\begin{center}
\begin{overpic}[scale=0.7]{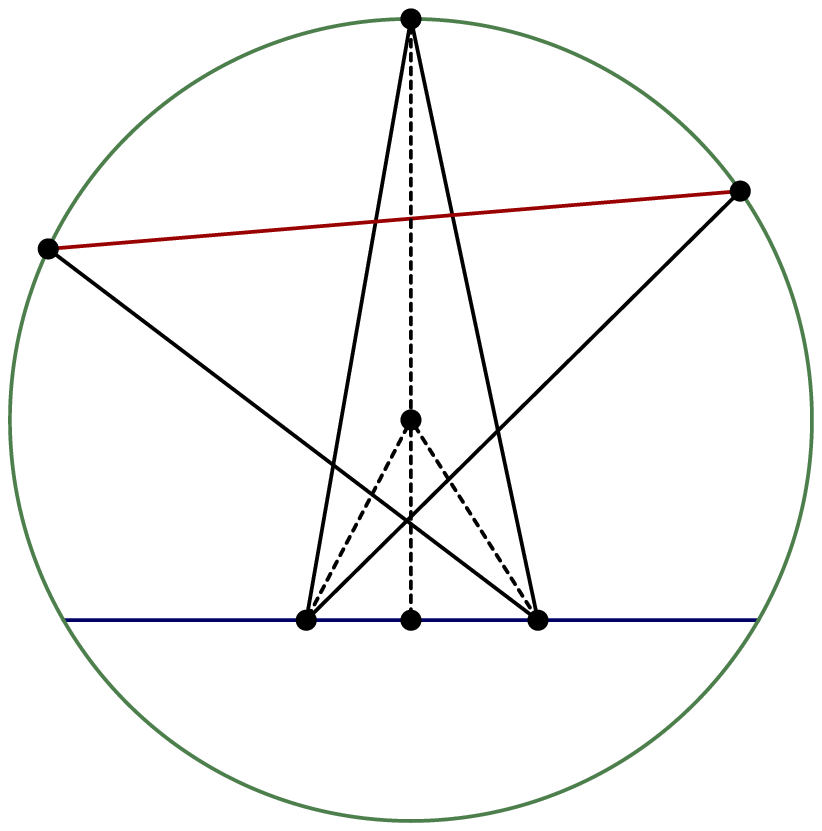}
\put(48,23){$o$}
\put(50,49){$o'$}
\put(60,23){$a$}
\put(36,23){$b$}
\put(0,71){$x_{n-1}$}
\put(84,71){$x_n$}
\put(47,93){$x_{n-2}$}
\put(73,30){$\ell$}
\put(14,84){$C'$}
\end{overpic}
\end{center}
\caption{Circle $C'$}\label{fig3}
\end{figure}
Therefore, $ao'$ bisects $\myangle x_{n-2}a x_{n-1}$, and $bo'$ bisects $\myangle x_{n-2} b x_{n-1}$.
Clearly, $\length{oa}>\length{ob}$, and both $\length{oa}$ and $\length{ob}$ are strictly monotone functions of $r$.

We now consider $r$ to be a variable ranging in the interval $(1/2,r_0)$, where
\[ r_0:=\left(2\cos\frac{\pi}{2(n-4)}\right)^{-1}.\]
On the one hand $r>\frac{1}{2}$, and in the limit as $r\to\frac{1}{2}$, the diameters $x_ix_{i+1}$ all coincide, and $\lim_{r\to1/2}\length{oa}=\lim_{r\to1/2}\length{ob}=0$.
It follows that
\[\lim_{r\to1/2}\length{x_{n-2}x_{n-1}}=\lim_{r\to1/2}\length{x_{n-2}x_n}=0,\]
hence $\lim_{r\to1/2}\length{x_{n-1}x_n}=0$.

On the other hand, $r< r_0$,
where in the limit as $r\to r_0$, $x_1$ and $x_{n-3}$ coincide, and the points form the vertex set of a regular $(n-4)$-gon.
Thus \[\lim_{r\to r_0}\length{oa}= r_0,\] \[\lim_{r\to r_0}\length{ob}\to 2r_0\sin\frac{\pi}{n-4},\] and \[\lim_{r\to r_0}\length{x_{n-2}a}=1.\]
Since $2r_0>1$, $\lim_{r\to r_0}x_{n-1}$ is a point below the chord $\ell$ of $C'$ through $a$ and $b$.
(Note that $\ell$ is a diameter of $C$).
Also, \[\lim_{r\to r_0}\length{x_2 a}= \lim_{r\to r_0}\length{x_{n-4} a}=1,\] hence $\lim_{r\to r_0}x_n=a$.
Since $x_{n-1}$ is lower than $x_n$ (because $\length{oa}>\length{ob}$), when $x_{n-1}$ reaches $\ell$, $x_n$ has not reached $\ell$ yet.
Since $\length{x_n b}=\length{x_{n-2}b}$, it follows that the chord $x_n b$ is below $o'$.
Since at this stage (with $x_{n-1}\in\ell$) the chord $bx_{n-1}$ is below $o'$, it follows that the chord $x_{n-1}x_n$ is below $o'$.
Thus before $x_{n-1}$ reaches $\ell$, there is a stage where $x_{n-1}x_n$ passes through $o'$ with both $x_{n-1}$ and $x_n$ still above $\ell$, and therefore at distance at most $1$ to $x_{n-2}$.
From $s>r>1$ it follows that $\length{x_{n-1}x_n}>1$.
Since $\lim_{r\to1/2}\length{x_{n-1}x_n}= 0$, at some stage $\length{x_{n-1}x_n}<1$.
Therefore, at some inbetween stage, $\length{x_{n-1}x_n}=1$.
This finishes the construction for odd $n\geq 7$.
\end{proof}

We remark that the exception $n\neq 5$ in Lemma~\ref{lowdim}\eqref{e} is necessary.
Suppose there exist $5$ points on a $2$-sphere with $8$ diameters.
Then one of the points must be incident to $4$ diameters.
The other $4$ points are then concyclic, and among them there can be at most $3$ diameters (Lemma~\ref{lowdim}\eqref{b}), a contradiction.
On the other hand, it is easy to find $5$ points on a sphere with $7$ diameters.

The next lemma is well known.
We omit the easy proof.

\begin{lemma}\label{geometry}
Let $A$ and $B$ be finite subsets of $\R^d$, each of size at least $3$.
If $\length{ab}=1$ for all $a\in A$, $b\in B$, then the affine subspaces spanned by $A$ and $B$ are orthogonal, $A$ and $B$ lie on spheres of radii $r_a$ and $r_b$, say, such that $r_a^2+r_b^2=1$, and with common centre the point of intersection of the two subspaces.
\end{lemma}

\section{Optimised Lenz configurations}\label{optimised}
\subsection{Even dimensions \boldmath$d\geq 6$\unboldmath}
We have already defined a Lenz configuration in the introduction.
For any Lenz configuration $S$ on $n$ points lying on $p=d/2$ mutually orthogonal circles $C_i$ with centre $o$ and radius $1/\sqrt{2}$, we define $S_i:=S\cap C_i$ and $n_i:=\card{S_i}$.

\subsubsection{Unit distances}
Define
\[ u_d^L(n)=\max\{u(S): \text{$S$ is a Lenz configuration of $n$ points in $\R^d$}\}. \]
We call any Lenz configuration $S$ of $n$ points in $\R^d$ for which $u(S)=u_d^L(n)$ an \define{optimised Lenz configuration} (for unit distances).

\begin{proposition}
Let $d\geq 6$ be even, $n\geq 1$, $p=d/2$ and $n\equiv r\pmod{2d}$, $0\leq r\leq 2d-1$.
Then
\begin{align*}
u_d^L(n) &= \begin{cases}
t_p(n)+n-r & \text{if}\quad 0\leq r\leq p-1,\\
t_p(n)+n-p & \text{if}\quad p\leq r\leq 3p-1,\\
t_p(n)+n-2d+r & \text{if}\quad 3p\leq r\leq 4p-1,
\end{cases}
\end{align*}
\end{proposition}

\begin{proof}
Consider an optimised Lenz configuration $S$ on $p$ pairwise orthogonal circles $C_1,\dots,C_p$.
We may rearrange the points on each circle without changing the number of unit distances between circles.
By Lemma~\ref{lowdim}\eqref{a} and maximality, each $u(S_i)=n_i$ if $n_i\equiv 0\pmod{4}$ and $u(S_i)=n_i-1$ otherwise.
The problem is now that of maximising the function
\[ u(n_1,\dots,n_p):= \sum_{1\leq i<j\leq p}n_in_j + n - p + k(n_1,\dots,n_p),\]
over all nonnegative $n_1,\dots,n_p$ that sum to $n$, where $k(n_1,\dots,n_p)$ equals the number of $n_i$ divisible by $4$.
This easy but tedious exercise finishes the proof.
\end{proof}

\subsubsection{Diameters}
Define
\[
\begin{split}
M_d^L(n)=\max\{u(S): S & \text{ is a diameter $1$ Lenz configuration}\\
&\text{ of $n$ points in $\R^d$}\}.
\end{split}
\]
We call any diameter $1$ Lenz configuration $S$ of $n$ points in $\R^d$ for which $u(S)=M_d^L(n)$ an \define{optimised Lenz configuration} (for diameters).

\begin{proposition}
Let $d\geq 6$ be even, $n\geq d$, and $p=d/2$. Then
\[ M_d^L = t_p(n) + p.\]
\end{proposition}

\begin{proof}
Consider an optimised Lenz configuration $S$ of diameter $1$ on $p$ pairwise orthogonal circles $C_1,\dots,C_p$.
By Lemma~\ref{lowdim}\eqref{c},
each $u(S_i)\leq 1$.
Therefore, $u(S)\leq t_p(n)+p$.
Equality is clearly possible if $n\geq d$, by dividing the $n$ points as equally as possible between the $p$ circles, and ensuring that a diameter occurs within each $S_i$.
\end{proof}

\subsection{The dimension \boldmath$d=4$\unboldmath}
For any Lenz configuration $S$ on $n$ points lying on orthogonal circles $C_1$ and $C_2$ with common centre $o$ and radii $r_1$ and $r_2$ such that $r_1^2+r_2^2=1$, define $S_i:=S\cap C_i$ and $n_i:=\card{S_i}$.

\subsubsection{Unit distances}
This section is included for the sake of completeness.
Define
\[ u_4^L(n)=\max\{u(S): \text{$S$ is a Lenz configuration of $n$ points in $\R^4$}\}. \]
As shown by Brass \cite{Brass} and Van Wamelen \cite{VanWamelen}:

\begin{proposition}
Let $n\geq 5$.
Then
\begin{align*}
u_4^L(n) &= \begin{cases}
t_2(n)+n & \text{if $n$ is divisible by $8$ or $10$,}\\
t_2(n)+n-1 & \text{otherwise.}
\end{cases}
\end{align*}
\end{proposition}

\subsubsection{Diameters}
Define
\[
\begin{split}
 M_4^L(n)=\max\{u(S): S &\text{ is a diameter $1$ Lenz configuration}\\
& \text{ of $n$ points in $\R^4$}\}.
\end{split}
\]
We call any diameter $1$ Lenz configuration $S$ of $n$ points in $\R^4$ for which $u(S)=M_4^L(n)$ an \define{optimised Lenz configuration} (for diameters).

\begin{proposition}
Let $n\geq 6$. Then
\[ M_4^L(n) = \begin{cases}
t_2(n)+\lceil n/2\rceil + 1 & \text{if}\quad n\not\equiv 3\pmod{4},\\
t_2(n)+\lceil n/2\rceil    & \text{if}\quad n\equiv 3\pmod{4}.
\end{cases}\]
\end{proposition}

\begin{proof}
Consider an optimised Lenz configuration $S$ of diameter $1$ on pairwise orthogonal circles $C_1$ and $C_2$.
Without loss of generality $r_1\leq r_2$.
We now apply Lemma~\ref{lowdim}\eqref{b}, \eqref{c}.
If $u(S_2)>1$, then $r_2\leq1/\sqrt{3}$ and $r_1\geq \sqrt{2/3}>r_2$, a contradiction.
Therefore, $u(S_2)\leq 1$.
Also, $u(S_1)\leq n_1$, and if $n_1$ is even, $u(S_1)\leq n_1-1$.
It follows that
\[ u(S) \leq \begin{cases}
n_1n_2 + n_1 + 1 & \text{if $n_1$ is odd,}\\
n_1n_2 + n_1     & \text{if $n_1$ is even.}
\end{cases}\]
By considering the four cases of $n$ modulo $4$, it is easily checked that the maximum over all nonnegative $n_i$ with $n_1+n_2=n$ is as in the statement of the theorem.
For $n\geq 6$ it is also easy to see that there are configurations that attain this maximum.
\end{proof}

\subsection{Odd dimensions \boldmath$d\geq 7$\unboldmath}\label{odd7}
We introduce the notion of a weak Lenz configuration.
Let $d\geq 7$ be odd, $p=(d-1)/2$, and consider any orthogonal decomposition $\R^d=V_0\oplus V_1\oplus\dots\oplus V_p$ with $\dim V_0=1$ and $\dim V_i=2$ ($i=1,\dots,p$).
For each $i=1,\dots,p$, let $\Sigma_i$ be the sphere in $V_0\oplus V_i$ with centre $o$ and radius $1/\sqrt{2}$, and let $C_i$ be the circle in $V_i$ with centre $o$ and radius $1/\sqrt{2}$.
Let $p^+$ and $p^-$ be the two points in $V_0$ at distance $1/\sqrt{2}$ from $o$.
Then $p^+$ and $p^-$ are the north and south poles of each $\Sigma_i$ when $C_i$ is considered to be its equator.

Let $i\neq j$.
If some $x\in\Sigma_i$ is at unit distance to some point of $\Sigma_j\setminus C_j$, then $x$ is at unit distance to all of $\Sigma_j$ (since it is already at unit distance to $C_j$).
By Lemma~\ref{geometry}, $x\in C_i$.
It follows that no point of $\Sigma_i\setminus C_i$ can be at unit distance to a point of $\Sigma_j\setminus C_j$.

A \define{strong Lenz configuration} of $n$ points in $\R^d$ is a  translate of a finite subset of $C_1\cup\dots\cup C_{p-1}\cup\Sigma_p$ for some orthogonal decomposition.
(This is merely the odd-dimensional ``Lenz configuration'' of Section~\ref{section2}.)
A \define{weak Lenz configuration} of $n$ points in $\R^d$ is a translate of a finite subset of a $\Sigma_1\cup\dots\cup\Sigma_p$ for some orthogonal decomposition.
Strong Lenz configurations are clearly weak.
If $S$ is a weak Lenz configuration, we assume without loss of generality that it is a subset of $\Sigma_1\cup\dots\cup\Sigma_p$, and we define $S_i:=S\cap\Sigma_i\setminus\{p^+,p^-\}$ ($i=1,\dots,p$), $S_0:=S\cap\{p^+,p^-\}$, $n_i:=\card{S_i}$ ($i=0,\dots,p$), $n:=\card{S}$.

\subsubsection{Unit distances}
Define
\[ u_d^L(n)=\max\{u(S): \text{$S$ is a weak Lenz configuration of $n$ points in $\R^d$}\}. \]
We call any weak Lenz configuration $S$ of $n$ points in $\R^d$ for which $u(S)=u_d^L(n)$ an \define{optimised Lenz configuration} (for unit distances).
Unlike the even-dimensional case we cannot give an expression for $u_d^L(n)$ more accurate than the estimate $u_d^L(n)=t_p(n)+\Theta(n^{4/3})$ due to Erd\H{o}s and Pach \cite{ErdosPach}.
However, we next show that an optimised Lenz configuration must be strong for $n$ sufficiently large, depending on $d$.
This implies that $u_d^L(n)$ can be determined if the function $f(n)$, which gives the maximum number of unit distances for $n$ points on a $2$-sphere of radius $1/\sqrt{2}$, is known.

\begin{proposition}
For each odd $d\geq 7$ there exists $N(d)$ such that all optimised Lenz configurations for unit distances on $n\geq N(d)$ points in $\R^d$ are strong Lenz configurations.
\end{proposition}

\begin{proof}
Let $S$ be an optimised Lenz configuration on $n$ points.
Suppose $S$ is not a strong Lenz configuration.
We aim for a contradiction.

Without loss of generality $S_i\setminus C_i\neq\emptyset$ for $i=1, 2$.
Since $u(S_1\setminus C_1)=O(\card{S_1\setminus C_1}^{4/3})$ (Lemma~\ref{lowdim}\eqref{new}) and $S_1\setminus C_1\neq\emptyset$, there exists $x\in S_1\setminus C_1$ with $u(x,S_1\setminus C_1)=O(\card{S_1\setminus C_1}^{1/3})=O(n^{1/3})$.
Also, since $x\neq p^\pm$, $u(x,C_1)\leq 2$.
Therefore, $u(x,S_1)=O(n^{1/3})$.
Note that for each $i=2,\dots,p$, $x$ is at distance $1$ to all points in $S_i\cap C_i$, but to none of $S_i\setminus C_i$.
If we replace $x$ by a new point on $C_1$, we lose at most $u(x,S_1)$ unit distances and gain $\sum_{i=2}^p\card{S_i\setminus C_i}$.
Since $u(S)$ is the maximum over all weak Lenz configurations,
\[\sum_{i=2}^p\card{S_i\setminus C_i}\leq u(x,S_i)=O(n^{1/3}).\]

By instead considering a point $x\in S_2\setminus C_2$ we obtain similarly that
\[\sum_{\substack{i=1\\ i\neq 2}}^p\card{S_i\setminus C_i}=O(n^{1/3}).\]
Therefore, $\card{S_i\setminus C_i}=O(n^{1/3})$ for each $i=1,\dots,p$.

We can now bound $u(S)$ from above.
First note that each point of $S_0$ is at unit distance to all of $C_i$ and none of $\Sigma_i\setminus C_i$, each point of $\Sigma_i\setminus\{p^+,p^-\}$ is at unit distance to at most two points of $C_i$,
and $u(S_i\cap C_i)\leq\card{S_i\cap C_i}$ (Lemma~\ref{lowdim}\eqref{a}).
This gives:
\begin{align*}
u(S_i) &\leq u(S_0\cup S_i)\\
&= u(S_0,S_i)+u(S_i\cap C_i) + u(S_i\cap C_i, S_i\setminus C_i) + u(S_i\setminus C_i)\\
&\leq 2\card{S_i\cap C_i}+\card{S_i\cap C_i} +2\card{S_i\setminus C_i}+O(\card{S_i\setminus C_i}^{4/3})\\
&= O(n) + O((n^{1/3})^{4/3}) = O(n).
\end{align*}
Therefore,
\begin{align*}
u(S) &\leq t_p(n) + u(S_0\cup S_1) + \sum_{i=2}^p u(S_i)\\
&= t_p(n)+O(n),
\end{align*}
contradicting $u(S)=u_d^L(n)=t_p(n)+\Theta(n^{4/3})$ for large $n$.
\end{proof}

\subsubsection{Diameters}
Define
\[
\begin{split}
M_d^L(n)=\max\{u(S): S &\text{ is a diameter $1$ weak Lenz configuration}\\
&\text{ of $n$ points in $\R^d$}\}.
\end{split}
\]
We call any diameter $1$ weak Lenz configuration $S$ of $n$ points in $\R^d$ for which $u(S)=M_d^L(n)$ an \define{optimised Lenz configuration} (for diameters).

We show, exactly as the unit distance case, that an optimised Lenz configuration must be strong for large $n$, and determine the exact value of $M_d^L(n)$.

\begin{proposition}
For each odd $d\geq 7$ there exists $N(d)$ such that all optimised Lenz configurations for diameters on $n\geq N(d)$ points in $\R^d$ are strong Lenz configurations.
Furthermore,
\[ M_d^L(n) = t_p(n)+\left\lceil\frac{n}{p}\right\rceil+p-1=t_p(n-1)+n-1+p.\]
\end{proposition}

\begin{proof}
Choose a set $S$ of $n$ points equally distributed between the orthogonal circles $C_1,\dots,C_{p-1}$ and $2$-sphere $\Sigma_p$ such that the diameter of each $S\cap C_i$ is $1$ and furthermore $\card{S\cap\Sigma_p}=\lceil n/p\rceil$, $\card{S\cap C_p}=\lceil n/p\rceil-1$ and $p^+\in S$.
Then clearly $u(S)=t_p(n)+\lceil n/p\rceil +p-1$.
Therefore, $M_d^L(n)\geq t_p(n)+\lceil n/p\rceil +p-1$.
We need this lower bound in a moment.

Now let $S$ be any optimised Lenz configuration on $n$ points.
Let $k_i:=\card{S_i\setminus C_i}$ ($i=1,\dots,p$).
We have to show that $S$ is a strong Lenz configuration, i.e., that $k_i=0$ for all $i=1,\dots,p$ except at most one.

First consider the case where $S_0\neq\emptyset$, where without loss of generality, $S_0=\{p^+\}$.
Then
\begin{align*}
u(S) &= u(S\setminus\{p^+\})+\sum_{i=1}^p u(p^+,S_i)\\
&= \sum_{1\leq i<j\leq p} u(S_i,S_j)+\sum_{i=1}^p u(S_i) + \sum_{i=1}^p u(p^+,S_i)\\
&= \sum_{1\leq i<j\leq p} u(S_i,S_j)+ \sum_{i=1}^p u(S_i\cup\{p^+\})\\
&= \sum_{1\leq i<j\leq p} \card{S_i}\card{S_j} - \sum_{1\leq i<j\leq p} k_i k_j + \sum_{i=1}^p u(S_i\cup\{p^+\})\\
&\leq t_p(n-1) - \sum_{1\leq i<j\leq p} k_i k_j + \sum_{i=1}^p (n_i+1)\\
&= t_p(n-1) - \sum_{1\leq i<j\leq p} k_i k_j + n-1+p\\
&= t_p(n) + \left\lceil\frac{n}{p}\right\rceil +p-1 - \sum_{1\leq i<j\leq p} k_i k_j.
\end{align*}
Since $u(S)=M_d^L(n)\geq t_p(n)+\lceil n/p\rceil +p-1$, we obtain $\sum_{1\leq i<j\leq p} k_i k_j = 0$, which implies that $k_i=0$ for all $i$ except one.
This proves the theorem for the case $S_0\neq\emptyset$.

Next consider the case where $S_0=\emptyset$.
Without loss of generality $S_1\setminus C_1\neq\emptyset$, otherwise $u(S)\leq t_p(n)+p$, a contradiction.
By Lemma~\ref{lowdim}\eqref{e}, $u(S_1)\leq n_1$.
If we remove the points in $S_1$ and replace them by placing $p^+$ into $S_0$ and placing $n_1-1$ points of diameter $1$ on $C_1$ to form another set $S'$ of diameter $1$, then we lose at most $n_1$ diameters and gain $n_1+k_1\sum_{i=2}^p k_i$.
By maximality, $\sum_{i=2}^p k_i=0$, i.e., the original $S$ was already a strong Lenz configuration 
and $u(S)=u(S')$.
We have already shown that an optimised Lenz configuration that contains $p^+$ satisfies $u(S')=t_p(n) + \left\lceil\frac{n}{p}\right\rceil +p-1$.
This finishes the case $S_0=\emptyset$.
\end{proof}

\subsection{The dimension \boldmath$d=5$\unboldmath}
Consider an orthogonal decomposition $\R^5=V_0\oplus V_1\oplus V_2$ such that $\dim V_0=1$ and $\dim V_1=\dim V_2=2$.
Choose $r_1\in (0,1)$.
Let $\Sigma_1$ be the $2$-sphere in $V_0\oplus V_1$ with centre $o$ and radius $r_1$.
Let $C_2$ be the circle in $V_2$ with centre $o$ and radius $r_2:=\sqrt{1-r_1^2}$.
Then any point of $\Sigma_1$ and any point of $C_2$ are at unit distance.
We call a translate of a finite subset of $\Sigma_1\cup  C_2$ a \define{strong Lenz configuration} (equivalent to the $5$-dimensional ``Lenz configuration'' of Section~\ref{section2}).

To define a weak Lenz configuration takes more care than for odd $d\geq 7$.
Choose an additional parameter $r\in [0,r_1)$ and a point $o'\in V_0$ at distance $r$ to $o$.
Let $C_1$ be the circle with centre $o'$ and radius $s_1:=\sqrt{r_1^2-r^2}$ in the plane of $V_0\oplus V_1$ parallel to $V_1$ that passes through $o'$.
Let $\Sigma_2$ be the $2$-sphere in $V_0\oplus V_2$ with centre $o'$ and radius $s_2:=\sqrt{r_2^2+r^2}$.
Then $C_i\subset\Sigma_i$ ($i=1,2$) (Figure~\ref{fig4}).
\begin{figure}
\begin{center}
\begin{overpic}[scale=0.6]{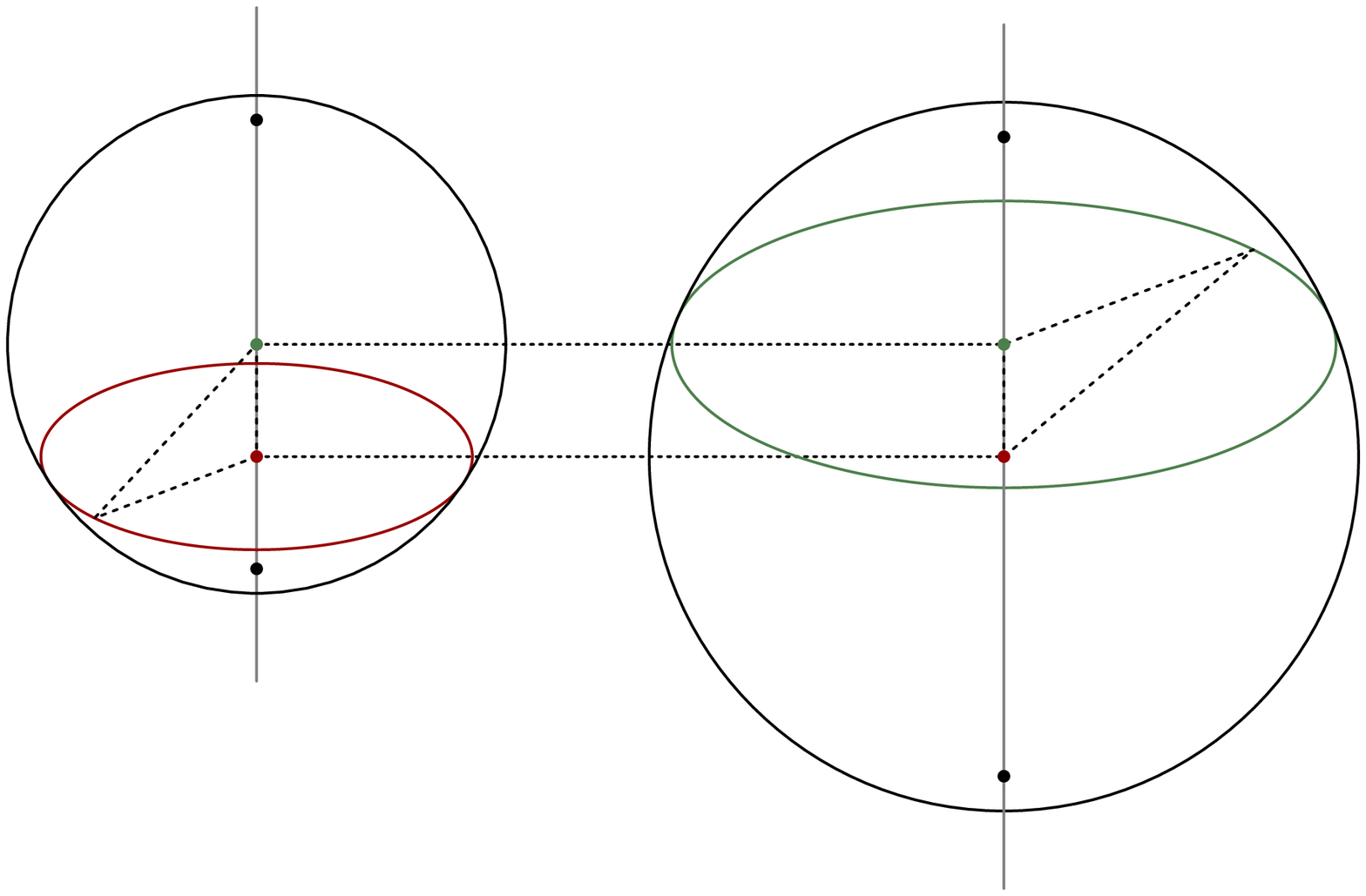}
\put(23,29){$o'$}
\put(23,40){$o$}
\put(23,60){$V_0$}
\put(1,60){$V_0\oplus V_1$}
\put(23,52){$p_1^+$}
\put(23,20){$p_1^-$}
\put(23,34.5){$r$}
\put(14,34){$r_1$}
\put(17,28){$s_1$}
\put(10,37.5){$C_1$}
\put(6,52){$\Sigma_1$}

\put(74,30.5){$o'$}
\put(70.5,40){$o$}
\put(74,59){$V_0$}
\put(86,59){$V_0\oplus V_2$}
\put(74,52){$p_2^+$}
\put(74,10){$p_2^-$}
\put(70.5,34.5){$r$}
\put(78,43.5){$r_2$}
\put(81,36){$s_2$}
\put(63.5,49.5){$C_2$}
\put(90,50){$\Sigma_2$}
\end{overpic}
\end{center}
\caption{Spheres $\Sigma_i$ and circles $C_i$ of a weak Lenz configuration in $\R^5$}\label{fig4}
\end{figure}
Note that $s_1^2+s_2^2=1$, hence any point of $\Sigma_2$ and any point of $C_1$ are at unit distance.
Similar to the discussion in Section~\ref{odd7} for odd $d\geq 7$, no point of $\Sigma_1\setminus C_1$ can be at unit distance to a point of $\Sigma_2\setminus C_2$.
We call a translate of a finite subset of $\Sigma_1\cup\Sigma_2$ a \define{weak Lenz configuration}.
As before, strong Lenz configurations are clearly weak.
Assume without loss of generality that $S\subset\Sigma_1\cup\Sigma_2$.
There are also \emph{poles}: $\{p_1^+, p_1^-\} := V_0\cap\Sigma_1$ and $\{p_2^+, p_2^-\} := V_0\cap\Sigma_2$.
In general, $\Sigma_1$ and $\Sigma_2$ may not have a point in common.
If they do, the common points will be coinciding poles.
Define $S_0:=S\cap V_0$, $S_i:=S\cap\Sigma_i\setminus V_0$ ($i=1,2$) and $n_i:=\card{S_i}$ ($i=0,1,2$), $n:=\card{S}$.

\subsubsection{Unit distances}
Define
\[ u_5^L(n)=\max\{u(S): \text{$S$ is a weak Lenz configuration of $n$ points in $\R^5$}\}. \]
We call any weak Lenz configuration $S$ of $n$ points in $\R^5$ satisfying $u(S)=u_5^L(n)$ an \define{optimised Lenz configuration} (for unit distances).
Again the best known estimate is $u_5^L(n)=t_2(n)+\Theta(n^{4/3})$, due to Erd\H{o}s and Pach \cite{ErdosPach}.
We show that an optimised Lenz configuration is strong for sufficiently large $n$.
As before, this implies that $u_5^L(n)$ can be determined if the function $g(n)$, which gives the maximum number of unit distances for $n$ points on a $2$-sphere of arbitrary radius, is known.

\begin{proposition}
For all sufficiently large $n$, all optimised Lenz configurations for unit distances on $n$ points in $\R^5$ are strong Lenz configurations.
\end{proposition}

\begin{proof}
Let $S$ be an optimised Lenz configuration on $n$ points.
Suppose that $S_1\setminus C_1\neq\emptyset$ and $S_2\setminus C_2\neq\emptyset$.
Then, using Lemma~\ref{lowdim}\eqref{new}, there exist points $x_i\in S_i\setminus C_i$ with $u(x_i,S_i\setminus C_i)=O(n^{1/3})$ ($i=1,2$).
Since $x_i\notin S_0$, $u(x_i,C_i)\leq 2$.
Thus $u(x_i,S_i)=O(n^{1/3})$.
If we replace each $x_i$ by a new point on $C_i$, we lose at most $O(n^{1/3})$ unit distances and gain $\card{S_1\setminus C_1}+\card{S_2\setminus C_2}$.
Since $S$ is extremal, $\card{S_1\setminus C_1}+\card{S_2\setminus C_2}=O(n^{1/3})$.
We bound $u(S)$ from above as in the case of odd $d\geq 7$.
For each $i=1,2$:
\begin{align*}
u(S_i) &\leq u(S_i\cup S_0)\\
&= u(S_0)+u(S_0,S_i)+u(S_i\cap C_i)+u(S_i\cap C_i,S_i\setminus C_i)+u(S_i\setminus C_i)\\
&\leq 4+4\card{S_i\cap C_i}+\card{S_i\cap C_i}+2\card{S_i\setminus C_i}+O(\card{S_i\setminus C_i}^{4/3})\\
&= O(n),
\end{align*}
hence,
\begin{align*}
u(S) &= u(S_1,S_2)+u(S_0\cup S_1)+u(S_0\cup S_2)+u(S_0)+u(S_1)+u(S_2)\\
&\leq t_2(n)+O(n),
\end{align*}
contradicting $u(S)=t_2(n)+\Theta(n^{4/3})$.

Therefore, some $S_i\setminus C_i=\emptyset$, without loss of generality $S_2\setminus C_2=\emptyset$.
To show that $S$ is a strong Lenz configuration, it remains to show that $S_0\subset\Sigma_1$.
Suppose then without loss of generality that $p_2^+\in S_0$ and $p_2^+\notin\Sigma_1$.
Then $p_1^\pm\neq p_2^+$.
Since $p_1^+$ is at unit distance to all of $C_2$, and $p_1^+$ and $p_2^+$ are different points in $V_0$, it follows that $p_2^+$ is not at unit distance to any point in $S_2$.
If we replace $p_2^+$ by a new point on $C_2$, we lose at most one unit distance (possibly between $p_2^+$ and $p_2^-$), and gain $\card{S\cap\Sigma_1\setminus C_1}$ unit distances.
By extremality, $\card{S\cap\Sigma_1\setminus C_1}\leq 1$.
Therefore, except for at most $3$ points (in addition, $p_2^+\in S_0$ and possibly $p_2^-\in S_0$), $S$ is on two orthogonal circles, and for this essentially $4$-dimensional configuration we obtain $u(S)\leq t_2(n)+O(n)$ as before, a contradiction.

It follows that $S$ is a strong Lenz configuration.
\end{proof}

\subsubsection{Diameters}
Define
\[
\begin{split}
M_5^L(n)=\max\{u(S): S &\text{ is a diameter $1$ weak Lenz configuration}\\
&\text{ of $n$ points in $\R^5$}\}.
\end{split}
\]
We call any diameter $1$ weak Lenz configuration $S$ of $n$ points in $\R^5$ satisfying $u(S)=M_5^L(n)$ an \define{optimised Lenz configuration} (for diameters).
Again an optimised Lenz configuration is strong for large $n$, and  the exact value of $M_5^L(n)$ can be determined.
However, this case is more intricate than odd $d\geq 7$.

\begin{proposition}
For all sufficiently large $n$, all optimised Lenz configurations for diameters on $n$ points in $\R^5$ are strong Lenz configurations.
Furthermore, $M_5^L(n)=t_2(n)+n$.
\end{proposition}

\begin{proof}
We first describe two types of strong Lenz configurations on $n$ points with $t_2(n)+n$ diameters.

In the first construction, choose $r_1$ such that there exists a set $S_1$ of $n_1$ points of diameter $1$ on $\Sigma_1$ with $2n_1-2$ diameters.
By Lemma~\ref{lowdim}\eqref{d} this is possible if $n_1\geq 4$, $n_1\neq 5$.
Choose any set $S_2$ of $n_2=n-n_1$ points of diameter $1$ on $C_2$.
(Note that $r_1<1/\sqrt{2}$ by Lemma~\ref{lowdim}\eqref{e}, which gives $r_2>1/\sqrt{2}>1/\sqrt{3}$.
Then by Lemma~\ref{lowdim}\eqref{c}, we can have at most one diameter of length $1$ on $C_2$.)
Let $S:=S_1\cup S_2$.
Then
\begin{align*}
u(S) &= u(S_1,S_2) +u(S_1)+u(S_2)\\
&= n_1 n_2+2n_1-2+1 = n_1(n_2+2)-1\\
&\leq t_2(n+2)-1=t_2(n)+n.
\end{align*}
Equality is possible by taking $n_1=\lfloor n/2\rfloor+1$ or $\lceil n/2\rceil+1$.
Keeping in mind that $n_1\geq 4$, $n_1\neq 5$, we obtain $t_2(n)+n$ diameters for all $n\geq 6$, $n\neq 8$.

In the second construction, first choose $r_2$ such that there exists a set $S_2$ of $n_2$ points of diameter $1$ on $C_2$ with $n_2$ diameters (a regular star polygon).
By Lemma~\ref{lowdim}\eqref{b} this is possible if $n_2\geq 3$ is odd.
Then $r_2\leq1/\sqrt{3}$ by Lemma~\ref{lowdim}\eqref{c}, and $r_1\geq\sqrt{2/3}>1/\sqrt{2}$.
By Lemma~\ref{lowdim}\eqref{e} we can then choose a set $S_1$ of $n_1=n-n_2$ points of diameter $1$ on $\Sigma_1$ with $n_1$ diameters if $n_1\geq 3$.
Let $S:=S_1\cup S_2$.
Then
\begin{align*}
u(S) &= u(S_1,S_2) +u(S_1)+u(S_2)\\
&= n_1 n_2+n_1+n_2= (n_1+1)(n_2+1)-1\\
&\leq t_2(n+2)-1=t_2(n)+n.
\end{align*}
Equality is possible by taking $n_1=\lfloor n/2\rfloor$, $n_2=\lceil n/2\rceil$ or $n_1=\lceil n/2\rceil$, $n_2=\lfloor n/2\rfloor$.
Keeping in mind the requirements that $n_2\geq 3$ must be odd and $n_1\geq 3$, we obtain $t_2(n)+n$ diameters for all $n\geq 6$, $n\not\equiv 0\pmod{4}$.
(It is because this second, simpler construction does not work for all $n$ that we need the construction in Lemma~\ref{lowdim}\eqref{d} of an odd number $n_1$ of points on a $2$-sphere with $2n_1-2$ diameters.)

Summarizing, $M_5^L(n)\geq t_2(n)+n$ for all $n\geq 9$.
It is easy to see that all strong Lenz configurations with at least $t_2(n)+n$ diameters must be one of the above two constructions for sufficiently large $n$.
We now turn to weak Lenz configurations.

Let $S$ be an optimised Lenz configuration on $n$ points.
We distinguish between two cases.

\bigskip\textbf{First case: \boldmath $S\cap\Sigma_1\cap\Sigma_2\neq\emptyset$.\unboldmath}
Any point in $S\cap\Sigma_1\cap\Sigma_2$ must be a common pole of $\Sigma_1$ and $\Sigma_2$, say $p_1^+=p_2^+$.
Since this point is at distance $1$ to $C_1$ and $C_2$, it follows that $\length{p_1^+p_1^-}, \length{p_2^+p_2^-}>1$.
Therefore, $S\cap\Sigma_1\cap\Sigma_2$ contains only one point $p:=p_1^+=p_2^+$, at distance $1$ to both $C_1$ and $C_2$.
Let $k_i:=\card{S_i\setminus C_i}$ ($i=1,2$).
Then
\begin{align}
t_2(n)+n &\leq u(S) \notag\\
 &= u(S_1,S_2)+u(S_1\cup\{p\})+u(S_2\cup\{p\})\notag\\
&= n_1n_2-k_1k_2+u(S_1\cup\{p\})+u(S_2\cup\{p\}).\label{star1}
\end{align}

If $u(S_i\cup\{p\})\leq n_i+1$ for both $i=1,2$, then by substituting into \eqref{star1},
\begin{align*}
t_2(n)+n &\leq n_1n_2-k_1k_2+n_1+1+n_2+1\\
&= (n_1+1)(n_2+1) -k_1k_2+1\\
&\leq t_2(n+1)-k_1k_2+1 \qquad\text{(note $n_1+n_2+1=n$)}\\
&= t_2(n)+\left\lceil\frac{n}{2}\right\rceil-k_1k_2+1.
\end{align*}
Therefore, $\lfloor n/2\rfloor +k_1k_2\leq 1$, a contradiction.

Without loss of generality we may therefore assume that $u(S_1\cup\{p\})> n_1+1$.
By Lemma~\ref{lowdim}\eqref{e}, $r_1<1/\sqrt{2}$, which gives $r_2>1/\sqrt{2}$ and $u(S_2\cup\{p\})\leq n_2+1$ (again Lemma~\ref{lowdim}\eqref{e}).
Also, $u(S_1\cup\{p\})\leq 2(n_1+1)-2=2n_1$ (Lemma~\ref{lowdim}\eqref{d}).
Substituting into \eqref{star1},
\begin{align*}
t_2(n)+n &\leq n_1n_2-k_1k_2+2n_1+n_2+1\\
&= (n_1+1)(n_2+2) -k_1k_2-1\\
&\leq t_2(n+2)-k_1k_2-1\\
&= t_2(n)+n-k_1k_2.
\end{align*}
It follows that $k_1k_2=0$, $S$ is a strong Lenz configuration, and $u(S)=t_2(n)+n$.

\bigskip\textbf{Second case: \boldmath $S\cap\Sigma_1\cap\Sigma_2=\emptyset$.\unboldmath}
Then $S$ may still contain poles, but a pole of $\Sigma_i$ in $S$ is not at distance $1$ to $C_i$ (otherwise it would also be a pole of the other sphere).
We now define $T_i=S\cap\Sigma_i$ ($i=1,2$).
Then $T_1, T_2$ partition $S$ (and we forget about the partition $S_0, S_1, S_2$).
Let $m_i:=\card{T_i}$ and $k_i:=\card{T_i\setminus C_i}$ ($i=1,2$).
As in the first case,
\begin{align}
t_2(n)+n &\leq u(S) \notag\\
 &= u(T_1,T_2)+u(T_1)+u(T_2) \notag\\
&= m_1m_2-k_1k_2+u(T_1)+u(T_2). \label{star2}
\end{align}

If $u(T_i)\leq m_i$ for both $i=1,2$, then by substituting into \eqref{star2},
\begin{align*}
t_2(n)+n &\leq m_1m_2-k_1k_2+m_1+m_2\\
&= (m_1+1)(m_2+1) -k_1k_2-1\\
&\leq t_2(n+2)-k_1k_2-1\\
&= t_2(n)+n-k_1k_2.
\end{align*}
It follows that $k_1k_2=0$, $S$ is a strong Lenz configuration, and $u(S)=t_2(n)+n$.

Otherwise, without loss of generality, $u(T_1)>m_1$.
As in the first case,
\begin{equation}\label{three}
u(T_1)\leq 2m_1-2
\end{equation}
and
\begin{equation}\label{one}
u(T_2)\leq m_2.
\end{equation}
Since each point in $T_i\setminus C_i$ is joined to at most two points of $T_i\cap C_i$ (recall that in this case a pole is not joined to any point on $C_i$), we also obtain
\begin{align}
u(T_1) &= u(T_1\cap C_1)+u(T_1\cap C_1, T_1\setminus C_1)+u(T_1\setminus C_1) \notag\\
&\leq \card{T_1\cap C_1}+2\card{T_1\setminus C_1}+2\card{T_1\setminus C_1}-2 \notag\\
&= m_1+3k_1-2 \label{four}
\end{align}
and since $r_2>1/\sqrt{2}$,
\begin{align}
u(T_2) &= u(T_2\cap C_2)+u(T_2\cap C_2, T_2\setminus C_2)+u(T_2\setminus C_2) \notag\\
&\leq 1+2\card{T_2\setminus C_2}+\card{T_2\setminus C_2} \notag\\
&= 1+3k_2. \label{two}
\end{align}
Substituting \eqref{one} and \eqref{four} into \eqref{star2}:
\begin{align*}
t_2(n)+n &\leq m_1m_2-k_1k_2+m_1+3k_1-2+m_2\\
&= (m_1+1)(m_2+1) -k_1(k_2-3)-3\\
&\leq t_2(n+2)-k_1(k_2-3)-3\\
&= t_2(n)+n-k_1(k_2-3)-2.
\end{align*}
Therefore, $k_1(k_2-3)+2\leq 0$, hence $k_2\leq 2$.

Substituting \eqref{three} and \eqref{two} into \eqref{star2}:
\begin{align*}
t_2(n)+n&\leq m_1m_2-k_1k_2+2m_1-2+3k_2+1\\
&= m_1(m_2+2) -(k_1-3)k_2-1\\
&\leq t_2(n+2)-(k_1-3)k_2-1\\
&= t_2(n)+n-(k_1-3)k_2.
\end{align*}
Therefore, $(k_1-3)k_2\leq 0$.
If $k_2>0$, then $k_1\leq 3$, and substituting \eqref{four} and \eqref{two} into \eqref{star2}:
\begin{align*}
t_2(n)+n&\leq m_1m_2-k_1k_2+m_1+3k_1-2+3k_2+1\\
&= m_1(m_2+1)+O(1)\\
&\leq t_2(n+1)+O(1)\\
&= t_2(n)+\lceil\frac{n}{2}\rceil+O(1),
\end{align*}
a contradiction.
It follows that $k_2=0$, giving that $S$ is a strong Lenz configuration, and $u(S)=t_2(n)+n$.
\end{proof}

\section{Stability theorems}\label{stability}
We formulate the stability theorem of Erd\H{o}s and Simonovits \cite[Chapter 5, Theorem 4.2]{Bollobas} in the following convenient way.
Let $K_r(t)$ denote the complete $r$-partite graph with $t$ vertices in each class.

\begin{stabilitytheorem}
For any $p, t\geq 2$ and any $\epsi>0$ there exists $N$ and $\delta>0$ such that if $G$ is any graph with $n\geq N$ vertices, at least $(\frac{p-1}{2p}-\delta)n^2$ edges and does not contain $K_{p+1}(t)$, then the vertices of $G$ can be partitioned into sets $S_0, S_1, \dots, S_p$ such that $\card{S_0}< \epsi n$, for each $i=1,\dots,p$,
\[\frac{n}{p}-\epsi n < \card{S_i} < \frac{n}{p}+\epsi n,\]
and each $x\in S_i$ is joined to all vertices of $G-S_i$ with the exception of less than  $\epsi n$.
\end{stabilitytheorem}

We now use the Stability Theorem to prove Theorems~\ref{evenstable} and \ref{oddstable}.

\begin{proof}[Proof of Theorem~\ref{evenstable}]
Without loss of generality $\epsi < 1/(3p^2)$.
By Lemma~\ref{geometry}, $K_{p+1}(3)$ does not occur in the unit distance graph of $S$.
Let $S_0,S_1,\dots,S_p$ be the partition coming from the Stability Theorem.
Suppose $S_1$ is not on a circle.
Let $A_1$ be a set of $4$ nonconcyclic points of $S_1$.
For each $i=2,\dots,p$, let $A_i$ consist of $3$ points of $S_i$ such that any two points in distinct $A_i$'s are joined.
This is possible, since each $x\in S_i$ is at unit distance to all points in $S\setminus S_i$ except for $\epsi n$ points, and $(4+3(p-2))\epsi n + 3 < n/p -\epsi n$ if $n>9p^2$.
The unit distance graph of $\bigcup_{i=1}^p A_i$ contains a complete $p$-partite graph with $4$ vertices in one class, and $3$ vertices in each other class.
By Lemma~\ref{geometry}, each $A_i$ is concyclic, a contradiction.

Therefore, each $S_i$ ($i=1,\dots,p$) is concyclic.
To see that these circles are orthogonal, choose $3$ points from each $S_i$ as above to form a $K_p(3)$.
Again by Lemma~\ref{geometry} each class lies on a circle $C_i$, with $C_1,\dots,C_p$ mutually orthogonal.
Since there is a unique circle through any $3$ noncollinear points, $S_i\subset C_i$ for each $i=1,\dots,p$.
\end{proof}

The following is the even-dimensional case of Corollary~\ref{stablecor}.

\begin{corollary} Fix an even $d\geq 4$.
If a set $S$ of $n$ points in $\R^d$ has at least $(\frac{p-1}{2p}-o(1))n^2$ unit distance pairs, then $S$ is a Lenz configuration except for $o(n)$ points.
\end{corollary}

\begin{proof}[Proof of Theorem~\ref{oddstable}]
Without loss of generality, $\epsi < 1/(4p^2)$.
By Lemma~\ref{geometry}, $K_{p+1}(3)$ does not occur in the unit distance graph of $S$.
Let $S_0,S_1,\dots,S_p$ be the partition coming from the Stability Theorem using $\epsi'=\epsi/5$.
Suppose $S_1$ is not on a $2$-sphere.
Let $A_1$ be a set of $5$ points of $S_1$ that are not contained in any sphere.
For each $i=2,\dots,p$, let $A_i$ consist of $3$ points of $S_i$ such that any two points in distinct $A_i$'s are joined.
This is possible, since each $x\in S_i$ is at unit distance to all points in $S\setminus S_i$ except for $\epsi' n$ points, and $(5+3(p-2))\epsi' n + 3 < n/p -\epsi' n$ if $n>4p$.
The unit distance graph of $\bigcup_{i=1}^p A_i$ contains a complete $p$-partite graph with $5$ vertices in one class, and $3$ vertices in each other class.
By Lemma~\ref{geometry}, each $A_i$ is on a sphere, a contradiction.

Therefore each $S_i$ ($i=1,\dots,p$) is on a $2$-sphere.
If each $S_i$ lies on a circle, then as in the even-dimensional case it follows that these circles are orthogonal.
Without loss of generality, $S_1$ is not concyclic.
Let $\Sigma_1$ denote the $2$-sphere on which $S_1$ lies.
Let $A_1$ be a set of $4$ noncoplanar points of $S_1$.

We now modify the partition of $S$ slightly.
There are less than $4\epsi'n$ points of $\bigcup_{i=2}^p S_i$ not joined to all of $A_1$.
Remove these points from $\bigcup_{i=2}^p S_i$ and add them to $S_0$.
Thus we me assume that each point of $A_1$ is joined to all of $\bigcup_{i=2}^p S_i$, but now we only have $\card{S_0}<5\epsi'n=\epsi n$, for each $i=1,\dots,p$, $\abs{\card{S_i}-n/p}<\epsi n$, and each point of $S_i$ is joined to less than $\epsi n$ points of $S\setminus S_i$.
We show that for this modified partition, $S_2,\dots,S_p$ are on circles $C_2,\dots,C_p$, with $\Sigma_1,C_2,\dots,C_p$ mutually orthogonal.

Suppose some $S_i$ ($i=2,\dots,p$) is not concyclic, without loss of generality $S_2$.
Let $A_2$ be $4$ nonconcyclic points from $S_2$, and as before, for $i=3,\dots,p$, let $A_i$ be $3$ points from $S_i$ such that all points in different $A_i$'s are joined.
By Lemma~\ref{geometry} the $A_i$ lie on spheres in mutually orthogonal subspaces.
By choice of $A_1$ it spans a $3$-dimensional space.
Since $A_2$ is cospherical but not concyclic, it also spans a $3$-dimensional space.
The other $A_i$ each spans at least $2$ dimensions.
We obtain at least $3+3+2(p-2)=d+1$ dimensions, a contradiction.

Therefore, each $S_i$ ($i=2,\dots,p$) is on a circle $C_i$.
As before, to see that $\Sigma_1,C_2,\dots,C_p$ are mutually orthogonal, choose $4$ noncoplanar points from $S_1$ and $3$ points from the other $S_i$ to form a complete $p$-partite graph, and apply Lemma~\ref{geometry}.
\end{proof}

The following is the odd-dimensional case of Corollary~\ref{stablecor}.

\begin{corollary} Fix an odd $d\geq 5$.
If a set $S$ of $n$ points in $\R^d$ has at least $(\frac{p-1}{2p}-o(1))n^2$ unit distance pairs, then $S$ is a strong Lenz configuration except for $o(n)$ points.
\end{corollary}

\section{Extremal sets are (weak) Lenz configurations}\label{extremal}
The following three results, completing the proof of the main theorem, follow relatively simply from the stability theorems.

\begin{proposition}\label{evenextremal}
For each even $d\geq 4$ there exists $N(d)$ such that all sets of $n\geq N(d)$ points in $\R^d$ extremal with respect to unit distances or diameters, are Lenz configurations.
\end{proposition}

\begin{proof}
When considering diameters assume that the diameter is $1$.
In both cases an extremal set $S$ on $n$ points has at least $\frac{p-1}{2p}n^2$ unit distances, so we may apply Theorem~\ref{evenstable} with $\epsi=1/(2p^2)$.
Thus for $n$ sufficiently large depending on $d$ we have a partition $S_0,S_1,\dots,S_p$ of $S$ with $\card{S_0}<\epsi n$ and for $i=1,\dots,p$, $\abs{\card{S_i}-n/p}<\epsi n$ and the $S_i$ are on orthogonal circles $C_i$.

We use the extremality of $S$ to show that $S_0\subset\bigcup_{i=1}^p C_i$.
Let $x\in S_0$.
If $u(x,S_i)\geq 3$ for all $i=2,\dots,p$, then by Lemma~\ref{geometry}, $x$ is on a circle of radius $1/\sqrt{2}$ in the plane orthogonal to the span of $\bigcup_{i=2}^p C_i$, i.e., $x\in C_1$.
Thus without loss of generality, $u(x,S_i)\leq 2$ for at least two $i$'s, say $i=1,2$.
Then
\begin{align*}
u(x,S) &=\sum_{i=0}^p u(x,S_i)\leq \card{S_0}-1+2+2+\sum_{i=3}^p\card{S_i}\\
&< \epsi n-1+4 + (p-2)(\frac{n}{p}+\epsi n)= \left(1-\frac{2}{p}+\epsi(p-1)\right)n+3.
\end{align*}
If we remove $x$ and replace it with a new point $x'\in C_1$, then
\begin{align*}
u(x',S\setminus\{x\}) &\geq u(x',\bigcup_{i=2}^p S_i)=\sum_{i=2}^p\card{S_i}\\
&> (p-1)\left(\frac{n}{p}-\epsi n\right)=\left(1-\frac{1}{p}-(p-1)\epsi\right)n.
\end{align*}
In the case of diameters we have to take care that $x'$ does not increase the diameter.
This can be done as follows.

Since all points of $C_1$ are already at unit distance to all points of $\bigcup_{i=2}^p C_i$, it is sufficient to choose $x'$ at distance at most $1$ to each point of $S_0$.
When $d\geq 6$, $C_1$ has radius $1/\sqrt{2}$, hence $S_1$ is contained in a $90^\circ$ arc $\gamma$ of $C_1$.
The set of points on $C_1$ at distance larger than $1$ from some $y\in S_0$ is a (perhaps empty) subarc of $\gamma$.
Such a subarc does not contain any point of $S_1$, and is therefore between some two consecutive points of $S_1$.
Since $\card{S_1}\geq \card{S_0}+1$ for $n$ sufficiently large, there exist two consecutive points of $S_1$, say $a$ and $b$, with no subarc between them.
Therefore, all points on $C_1$ between $a$ and $b$ are at distance at most $1$ to all points of $S_0$, and we may choose $x'$ to be any point on $C_1$ between $a$ and $b$.

When $d=4$, one of the two circles $C_1$ and $C_2$ has radius at least $1/\sqrt{2}$, and the above argument also works for this circle.

Since $S$ is extremal, such a modification cannot increase the number of unit distances:
\[ u(S)\geq u(S\cup\{x'\}\setminus\{x\}), \]
hence
\[ u(x,S)\geq u(x',S\setminus\{x\}), \]
i.e.,
\[ \left(1-\frac{2}{p}+\epsi(p-1)\right)n+3 > \left(1-\frac{1}{p}-\epsi(p-1)\right)n, \]
which is a contradiction if $\epsi=1/(2p^2)$ and $n\geq 3p^2$.
Therefore, $x\in C_1$.

We have shown that $S_0\subset\bigcup_{i=1}^p C_i$, which implies that $S$ is a Lenz configuration for large $n$.
\end{proof}

\begin{theorem}\label{oddextremal}
For each odd $d\geq 7$ there exists $N(d)$ such that all sets of $n\geq N(d)$ points in $\R^d$ extremal with respect to unit distances or diameters, are weak Lenz configurations.
\end{theorem}

\begin{proof}
Again in the case of diameters assume that the diameter is $1$.
An extremal set $S$ on $n$ points has at least $\frac{p-1}{2p}n^2$ unit distances, so we may apply Theorem~\ref{oddstable} with $\epsi=1/(4p^2)$.
Thus for $n$ sufficiently large depending on $d$ we have a partition $S_0,S_1,\dots,S_p$ of $S$ with $\card{S_0}<\epsi n$ and for $i=1,\dots,p$, $\abs{\card{S_i}-n/p}<\epsi n$, $S_1$ is on a sphere $\Sigma_1$, each $S_i$ ($i=2,\dots,p$) is on a circle $C_i$, and $\Sigma_1,C_2,\dots,C_p$ are mutually orthogonal and all have radius $1/\sqrt{2}$.

To show that $S$ is a weak Lenz configuration, it is sufficient to show that each point of $S_0$ not on $\Sigma_1$ lies on the $2$-sphere of radius $1/\sqrt{2}$ containing some $C_i$ ($i=2,\dots,p$) in the subspace generated by $C_i$ and some fixed diameter of $\Sigma_1$.

As in the proof of Theorem~\ref{evenextremal}, extremality of $S$ implies a lower bound on the degree of each point $x\in S$.
As before we find a point $x'\in C_2$ without increasing the diameter.
Since $S$ is extremal,
\begin{align}
u(x,S) &\geq u(x',S\setminus\{x\})\geq \sum_{\substack{i=1\\i\neq 2}}^p\card{S_i}\notag\\
&> (p-1)\left(\frac{n}{p}-\epsi n\right)=\left(1-\frac{1}{p}-(p-1)\epsi\right)n.\label{star}
\end{align}
For $i=2,\dots,p$ define
\[ T_i:= \{x\in S_0 : u(x,S_i)\leq 2\}.\]
Clearly for any point $x\in\Sigma_1$, $u(x,S_i)=\card{S_i}>\frac{n}{p}-\epsi n\geq 3$ for $n>4p$, and therefore $\bigcup_{i=2}^p T_i\subseteq S_0\setminus\Sigma_1$.
Conversely, if $x\in S_0$ and $u(x,S_i)\geq 3$ for each $i=2,\dots,p$, then $x\in\Sigma_1$ (Lemma~\ref{geometry}).
It follows that $\bigcup_{i=2}^p T_i=S_0\setminus\Sigma_1$.
We next show that $T_2,\dots,T_p$ partition $S_0\setminus\Sigma_1$.
If not, there exists $x\in S_0\setminus\Sigma_1$ with $u(x,S_i)\leq 2$ and $u(x,S_j)\leq 2$ for distinct $i,j\in\{2,\dots,p\}$.
Then
\begin{align*}
u(x,S) &= u(x,S_0)+u(x,S_1)+\sum_{i=2}^p u(x,S_i)\\
&< \epsi n +\frac{n}{p}+\epsi n+ 2+2+ (p-3)\left(\frac{n}{p}+\epsi n\right)\\
&= \left(1-\frac{2}{p}+(p-1)\epsi\right)n+4,
\end{align*}
which contradicts the lower bound \eqref{star} when $n>8p$.

Note that the neighbours in $S_1$ of an $x\in S_0\setminus\Sigma_1$ all lie on a circle $C_1$, say, of $\Sigma_1$.
We now show that this circle is the same for all $x\in S_0\setminus\Sigma_1$.
First we bound $u(x,S_1)$ from below:
\begin{align*}
u(x,S) &= u(x,S_0)+u(x,S_1)+\sum_{i=2}^p u(x,S_i)\\
&< \epsi n +u(x,S_1)+2+(p-2)\left(\frac{n}{p}+\epsi n\right)\\
&= u(x,S_1)+\left(1-\frac{2}{p}+(p-1)\epsi\right)n+2,
\end{align*}
which, together with the estimate \eqref{star}, gives
\[ u(x,S_1)>\left(\frac{1}{p}-2(p-1)\epsi\right)n-2. \]
If the neighbours in $S_1$ of some other $x'\in S_0\setminus\Sigma_1$ are on another circle of $\Sigma_1$, then
\[\card{S_1}\geq u(x,S_1)+u(x',S_1)-2 > 2\left(\frac{1}{p}-2(p-1)\epsi\right)n-6.\]
Since $\card{S_1}<\frac{n}{p}+\epsi n$, we have a contradiction if $n>8p^2$.

Therefore, the neighbours in $S_1$ of any $x\in S_0\setminus\Sigma_1$ are on $C_1$.
Since $C_1$ contains at least $3$ points of $S_1$, it is orthogonal to $C_2,\dots,C_p$ (Lemma~\ref{geometry}), and therefore it has radius $1/\sqrt{2}$, and is a great circlce of $\Sigma_1$.
For each $i=2,\dots,p$, let $\Sigma_i$ be the sphere of radius $1/\sqrt{2}$ which has $C_i$ as great circle, in the $3$-space containing $C_i$ and the diameter of $\Sigma_1$ perpendicular to $C_1$.
Since $T_2,\dots,T_p$ is a partition, each point of $T_i$ is at distance $1$ to at least $3$ points of each $C_j$, $j\neq i$, and by Lemma~\ref{geometry}, $T_i\subset\Sigma_i$.
Since also $S_i\subset\Sigma_i$, we have shown that $S$ is a weak Lenz configuration for large $n$.
\end{proof}

\begin{theorem}
For all sufficiently large $n$, all sets of $n$ points in $\R^5$ extremal with respect to unit distances or diameters are weak Lenz configurations.
\end{theorem}
\begin{proof}
An extremal set $S$ of $n$ points has at least $n^2/4$ unit distances, so by Theorem~\ref{oddstable} with $\epsi=1/11$ we obtain that for sufficiently large $n$, $S$ can be partitioned into $S_0,S_1,S_2$ such that $\card{S_0}<\epsi n$, $\abs{\card{S_i}-n/2}<\epsi n$ ($i=1,2$), $S_1$ is on a sphere $\Sigma_1$ of radius $r_1$, $S_2$ is on a circle $C_2$ of radius $r_2$, such that $\Sigma_1$ and $C_2$ are orthogonal and $r_1^2+r_2^2=1$.

As in the proof for odd $d\geq 7$, if $r_2\geq1/\sqrt{2}$, we can find a point $x'\in C_2$ that does not increase the diameter.
Otherwise, $r_1\geq1/\sqrt{2}$, and we consider the intersection of $\Sigma_1$ and all balls in the $3$-space of $\Sigma_1$ of radius $1$ centred at points in $S\cap\Sigma_1$.
This gives a spherically convex set on $\Sigma_1$ containing $S\cap\Sigma_1$.
Any new point $x'$ in this set is at distance at most $1$ to all points of $S$.
As before, replacing any point $x\in S$ by $x'$ gives $u(x,S)>(\frac{1}{2}-\epsi)n$.
Note that if $u(x,S_2)\geq 3$ for some $x\in S_0$, then $x\in\Sigma_1$.
Therefore, $u(x,S_2)\leq 2$ for all $x\in S_0\setminus\Sigma_1$.
Next we bound $u(x,S_1)$ from below for all $x\in S_0\setminus\Sigma_1$:
\begin{align*}
\left(\frac{1}{2}-\epsi\right)n &< u(x,S) = u(x,S_0)+u(x,S_1)+u(x,S_2)\\
&< \epsi n+u(x,S_1)+2,
\end{align*}
hence
\[ u(x,S_1) > \left(\frac{1}{2}-2\epsi\right)n-2. \]
The neighbours in $S_1$ of an $x\in S_0\setminus\Sigma_1$ lie on a circle $C_1$, say, of $\Sigma_1$.
If the neighbours of some other $x'\in S_0\setminus\Sigma_1$ lie on another circle of $\Sigma_1$, then
\begin{align*}
\frac{n}{2}+\epsi n &> \card{S_1} > u(x,S_1)+u(x',S_1)-2\\
&> (1-4\epsi)n - 6.
\end{align*}
Therefore, $5\epsi n> \frac{n}{2}-6$, a contradiction for $n$ sufficiently large.

Let the radius of $C_1$ be $s_1$.
By Lemma~\ref{geometry}, each $x\in S_0\setminus\Sigma_1$ lies on its complementary sphere $\Sigma_2$ of radius $s_2$, where $s_1^2+s_2^2=1$, and $C_2\subset\Sigma_2$.
We have shown that $S$ is a weak Lenz configuration for large $n$.
\end{proof}

\end{document}